\documentclass[12pt,reqno]{article}

\usepackage[usenames]{color}
\usepackage{amssymb}
\usepackage{graphicx}
\usepackage{amscd}

\usepackage[colorlinks=true,
linkcolor=webgreen,
filecolor=webbrown,
citecolor=webgreen]{hyperref}

\definecolor{webgreen}{rgb}{0,.5,0}
\definecolor{webbrown}{rgb}{.6,0,0}

\usepackage{color}
\usepackage{fullpage}
\usepackage{float}

\usepackage{graphics,amsmath,amssymb}
\usepackage{amsthm}
\usepackage{amsfonts}
\usepackage{latexsym}
\usepackage{epsf}

\usepackage{verbatim}
\usepackage{enumerate}

\usepackage{tikz}
\usepackage{caption}
\usepackage{capt-of}
\usepackage{forest}
\usepackage{skak}

\usepackage{mathtools}

\usepackage{array}

\usepackage[blocks]{authblk}

\makeatletter
\newcommand{\thickhline}{
    \noalign {\ifnum 0=`}\fi \hrule height 1pt
    \futurelet \reserved@a \@xhline
}
\newcolumntype{"}{@{\hskip\tabcolsep\vrule width 1pt\hskip\tabcolsep}}
\makeatother

\theoremstyle{plain}
\newtheorem{theorem}{Theorem}

\newtheorem{lemma}[theorem]{Lemma}
\newtheorem{proposition}[theorem]{Proposition}

\theoremstyle{definition}

\theoremstyle{remark}

\begin{document}

\title{SOS}
\author{Jasper Caravan}
\author{Michael Han}
\author{Andrew Kalashnikov}
\author{Ella Li}
\author{Alexander Meng}
\author{Vaibhav Rastogi}
\author{Gordon Redwine}
\author{Lev Strougov}
\author{Angela Zhao}
\affil{PRIMES STEP}
\author{Tanya Khovanova}
\affil{MIT}

\maketitle

\begin{abstract}
SOS is a game similar to tic-tac-toe. We study a variety of variations of it played on a 1-by-$n$ rectangle. On our journey we change the target string to SOO, then study all target strings containing SSS, then go back to finite strings and study SOSO. Then we create a variation with two target strings SSSS-OOOO. We continue with mis\`{e}re games. After that we look at different changes of the board: we wrap the board around, add dimension, and allow it to expand. We end with a variation that came from the ultimate tic-tac-toe.
\end{abstract}

\section{The Original Problem and Its Solution}

SOS is not only a cry for help but also the name of a game similar to tic-tac-toe and dots and boxes, born in the midst of a cry for help. The year was 1999, and many people feared a big computer crash as soon as 2000 arrived. In the past, many programmers didn't expect their programs to last long and used the date format with only the last two digits for a year. That meant that as soon as January 1, 2000 arrived, it would be counted as a past date, potentially creating all sorts of problems. SOS!

The game appeared as a USAMO problem in 1999, see also \cite{F}.

\textbf{Problem.}
The Y2K Game is played on a $1 \times 2000$ grid as follows. Two players, in turn, write either an S or an O in an empty square. The first player who produces three consecutive cells that spell SOS wins. The game is a draw if all cells are filled without producing SOS. Prove that the second player has a winning strategy.

The solution to this problem appears in many places. In particular, the Art of Problem Solving website \cite{AoPS} collects competition problems and their solutions.

It is not very surprising that in a Y2K problem, the grid length is 2000. But since everyone has forgotten about the Y2K scare nowadays, we assume we have a 1-by-$n$ grid for any $n$. Before discussing the solution, we start with some definitions.

\subsection{Definitions}

We assume that both players play optimally, and the squares are labeled $1$ through $n$. We will use the letter E to describe an empty cell.

We call a move \textit{safe} if the opponent can't win on the next move. We call a move \textit{unsafe} if the opponent can win on the next move. Each move is either safe or unsafe.

We call two blank cells separated by two characters \textit{a losing pair} if any move on one of these blank cells is unsafe. We use this terminology because, when a losing pair is created, someone is guaranteed to lose, as we see in the next proposition.

\begin{proposition}
If there exists a losing pair on the board, the game is not a draw.
\end{proposition}

\begin{proof}
If a letter is placed inside the losing pair, the next player wins.
\end{proof}

We call a player who sees an odd number of cells before their turn a \textit{designated winner}. In other words, on odd-sized boards, the first player is a designated winner, and on even-sized boards, the second player is a designated winner.

\begin{theorem}
\label{thm:losingpaircreated}
If a losing pair is created without making unsafe moves, and a designated winner always has a safe move outside of losing pairs, then the designated winner wins.
\end{theorem}

\begin{proof}
Eventually, the player who sees an even number of empty squares has to use a cell in a losing pair and lose.
\end{proof}

Thus, the strategy for a designated winner is as follows: 
\begin{itemize}
\item Create a losing pair.
\item Find an algorithm to always have a safe move.
\end{itemize}

We call a situation when a board is one letter away from a win \textit{a brink}. This notion seems to be trivial, as in the original game, after a brink is created, the next player to move will win. However, this term is useful in variations of the original game. Creating a brink is an unsafe move.

\subsection{The solution to the original game}

In this case, a pair of empty cells between two S's is a losing pair. Indeed, if one places a letter inside a losing pair, then the other letter placed on the other cell is a win.

The following is the safe move algorithm for a designated winner.
\begin{itemize}
 \item   If possible, the player completes SOS on their turn.
 \item   Otherwise, since there are an odd number of empty cells at the beginning of the first player's turn, there must be either one empty square surrounded by two filled squares or one empty square surrounded by two empty squares. In both cases, the player places an O in that square --- it is easy to verify that the move is safe.
\end{itemize}

Now we have the final answer.

\begin{theorem}
\label{thm:SOS}
In the SOS game, the first player wins when $n \geq 7$ and odd. The second player wins when $n \geq 16$ and even. Otherwise, it is a draw.
\end{theorem}

\begin{proof}
Suppose $n \geq 7$ and odd. The first player places S in the middle. After the second player's move, the first player can win or create a losing pair by placing an S to the left or right of the middle cell. After the losing pair is created, the first player wins by Theorem~\ref{thm:losingpaircreated}.

Suppose $n \geq 16$ and even. If the first player puts an S, the second player can create a losing pair. Suppose the first player puts an O somewhere. There still exists a strip of seven empty cells not neighboring O. The second player places an S in the middle of this strip. The next time the second player moves, they can create the second S for the losing pair that doesn't neighbor the given O. Such a move doesn't create an immediate win for the first player. After the losing pair is created, the second player wins by Theorem~\ref{thm:losingpaircreated}.

For boards of sizes less than 7, we can manually check that they are a draw. We wrote a program that checked that the game is a draw for boards of size 8, 10, 12, and 14.
\end{proof}

The original game is solved, but what if we use a different target string, say SOO?

\section{SOO}

Suppose the target string is SOO. 

\begin{theorem}
\label{thm:SOO}
The SOO game is always a draw.
\end{theorem}

\begin{proof}
There is a strategy to prevent the other player from winning, which we now describe. If there is a winning move, then make it. Otherwise, place an S in the rightmost empty cell. This move is safe: suppose the opponent can win by placing a letter to the left of this S, then the original player could have won on the previous move.
\end{proof}

Did we cover all possible strings of three letters? The string SOS is equivalent to OSO by swapping the letters. Similarly, SOO is equivalent to OSS. But, in addition, we can swap left and right sides. This means Theorem~\ref{thm:SOO} also covers the cases of SSO and OOS. 

There are two strings left: SSS and OOO. By symmetry, we need to study only one case, say SSS. We do even more: we study games where the target string contains SSS as a substring.

\section{A string containing SSS}

Surely, if we proved a string $T$ to be a draw, it should immediately follow that any string containing $T$ must be a draw. Not so fast!

Consider the game of SOOS, which contains the string SOO. Our strategy for SOO involves the fact that we can always place S in the rightmost empty cell, and this is a safe move. However, for SOOS, this is not the case. For example, if the right end of our board contains SOEE, placing S in the rightmost square is unsafe. This makes the strategy for SOO not applicable to SOOS. However, we conjecture that SOOS is a draw, but this requires a separate proof. Moreover, after experiments with our SOS solver available at github \cite{GR}, we expect all games with target strings longer than three letters to be drawn. However, we digress.

We want to show that a string containing SSS is a draw. The above argument explains that it is not enough to prove that the game SSS is a draw. However, it works the other way; if we prove that a string containing SSS is a draw, then the SSS game is a draw.

\begin{theorem}
A game with a target string containing SSS is a draw.
\end{theorem}

\begin{proof}
Each player has a strategy that prevents the other player from winning. More precisely, the strategy prevents the creation of three S's in a row.

We divide the grid into dominoes, where a \textit{domino} is defined as a pair of two consecutive cells. If there is an odd number of squares, let the leftmost square stand on its own. Here is the strategy. On their turn, if a player can win, then the player does so. Otherwise, if a domino exists with exactly one of the squares filled with a letter, the player fills the other square with a different letter. Otherwise, the player places an O anywhere.

After the player's turn, any S inside a domino has to be next to an O. Thus, if two consecutive S's exist in the leftmost part of the board, then the board should start as SSO. If two consecutive S's are somewhere else, they are part of the OSSO substring. 

Moreover, having an SES substring at any point in the game is impossible. Indeed, if a square is surrounded by two S's, it will be in at least one of the dominoes containing an S, and thus it has to contain an O. Therefore, there are no situations when one move can create an SSS substring. Thus, all games with one target string that contains SSS are drawn.
\end{proof}

We showed that target strings containing a run of length 3 of the same letter are a draw. But do we have other examples? Can we solve a string of length 4? Let's try SOSO.

\section{SOSO}

We start by reducing our work by half.

\begin{lemma}
\label{lemma:secondplayer}
To prove that a target string is always a draw, it is enough to prove that the second player can guarantee a draw.
\end{lemma}

\begin{proof}
Without loss of generality, we can assume that the first letter of the target string is S. The first player, on their first move, puts an O in the leftmost square. This letter can't be part of a winning string. This way, after the first move, the first player can follow the strategy for the second player, guaranteeing a draw on a board that is one cell shorter.
\end{proof}

To study the SOSO game, we introduce the notion of \textit{purging}, which we define as removing the leftmost two filled cells when these cells are filled in such a way that they cannot contribute to the creation of SOSO. This means we can remove the two cells with no consequence. For example, we can always purge OO. We can also purge SS and OS if they are followed by S. We can also purge SO followed by O. In other words, if the second and third letters are the same, we can purge the first two letters.

\begin{theorem}
The SOSO game is a draw.
\end{theorem}

\begin{proof}
By Lemma~\ref{lemma:secondplayer}, it is enough to show that the second player can guarantee a draw. 

Suppose the board size is even. The second player divides the board into dominoes. Whenever the first player places a letter inside a domino, the second player plays the same letter inside the same domino. At the end of the game, the board will consist of letters divided into pairs of doubles, so it can't contain SOSO.

Suppose the board is odd. The second player marks the leftmost cell and divides the rest into dominoes. If the first player puts a letter inside a domino, the second player repeats it. Whenever possible, the second player purges the two leftmost filled cells.

Suppose the first player puts a letter in the leftmost cell. If there are no more empty cells, the game is a draw. If the first domino to the right of the first cell is not empty, then it is a double, so we can purge the first two letters. If some other dominoes to the right are filled, we continue purging until we get one letter followed by an empty domino. Thus, we can assume that the leftmost domino is empty.

Suppose the first player puts an O in the leftmost cell, then the second player plays an O in the next cell to the right and can purge again. Suppose the first player puts an S in the leftmost cell. Then the second player places an O in the third cell. Now the board starts with SEO, followed by dominoes. If the first player places a letter in any empty domino, the second player plays the same letter in the same domino.

Suppose the first player places a letter in the second cell. If this letter is an O, then we can purge. If the letter placed in the second cell is S, then if the next domino to the right is filled, we can purge. If the next domino is empty, we have the beginning of the board as SSOEE. Then the second player plays an O in the fourth cell and can purge the first four letters. 

In every case, each second player's move is safe; thus, the game is a draw.
\end{proof}

We've tried a variety of strings as target strings so far. What if we have not one but two target strings?

\section{SSSS-OOOO}

Now we study a game where a player wins if and only if one of the strings SSSS or OOOO is achieved. We call this game SSSS-OOOO.

\begin{theorem}
The SSSS-OOOO game is a draw.
\end{theorem}

\begin{proof}
Each player has a strategy that prevents the other player from winning.

We divide the grid into dominoes. If there is an odd number of squares, let the leftmost square stand on its own. On their turn, if a player can win, then the player does so. Otherwise, if a domino exists with exactly one of the squares filled with a letter, the player fills the other square with a different letter.

If none of these are true, the player puts a letter in the rightmost available square according to the following rule. If the empty square is the rightmost square of the board, the player puts any letter there. Otherwise, the player puts a letter that differs from the letter to the right. Suppose this letter is X. Then, to the left of X, all dominoes are either empty or contain different letters. To the right of the X, all dominoes are filled. Moreover, if there exists a domino with both letters the same, the next letter to the right is different, or the border is to the right.

Now we prove that this strategy is guaranteed to be safe. Let's assume the last letter placed is X. The opponent cannot win without using the letter X in their move; otherwise, the current player would have won on the previous move. Suppose X was the first letter placed into a domino; then, the letter to the right of X either doesn't exist or is different from X. Thus, the only option is to use X and the three spaces to the left. There are three possibilities for these: EEEX, SOEX, and OSEX. All these cases are safe.

Suppose X is the second letter placed into a domino. If it is the left letter, it differs from the right one. Then, whether the domino on the left is empty or contains different letters, the opponent can't finish the target string on the next move. Suppose X is the right letter in a domino. Then to the right of it, there are three cases: a) an empty domino or the border; b) a filled domino with two distinct letters; c) a filled domino with two same letters, but in this case, the next letter must either not exist or be different. In all of these cases, the move is safe.

Thus, all games with the two target strings being SSSS and OOOO are draws.
\end{proof}

We've tried different strings. We've even tried two strings as a goal. But what if we change the goal of the game? What if the player who gets the target string loses? Such games are called mis\`{e}re games. In combinatorial game theory, mis\`{e}re games are often more difficult to analyze. Is this also true in our case?

\section{Mis\`{e}re games}

In this version, the person who first gets the target string loses. Let's try a simple example. If the target string is SSS, then a player can never lose by always playing an O. Now we try a more complicated example where our target string is SOS.

\begin{theorem}
The mis\`{e}re SOS game is a draw for $n \geq 3$.
\end{theorem}

\begin{proof}
Consider an empty cell. If the cell is surrounded by S's, we can place an S there without creating an SOS. Otherwise, we can place an O without creating an SOS. Thus, in any case, any player always has a non-losing move.
\end{proof}

It sounds as if avoiding a string is easy. Is it ever possible to lose? Theoretically, it could be.

A \textit{losing square} is a square in mis\`{e}re games where placing either S or O there creates the target string. For example, in the mis\`{e}re SO game, the empty square in SEO is a losing square. We leave it as an exercise for the reader to show that the SO mis\`{e}re game on a 1-by-6 board is a win for the first player.

Similar to the SO game, in the mis\`{e}re SOO game, the empty square in SEOO is a losing square. However, this game too is always a draw.

\begin{theorem}
With optimal play, the mis\`{e}re SOO game is a draw.
\end{theorem}

\begin{proof}
Here is the strategy. If the previous move was an O, and the cell to the left of this O is available, then place an S to the left of the O. Otherwise, place an S in the rightmost empty cell. Now we show that a player can't lose with this strategy.

As the player with this strategy always places an S, the only potential way to lose is if there exist two Os in a row with an empty cell on the left. However, one of these Os was played in one of the previous moves, which means it has to have an S on the left, proving that the strategy works.
\end{proof}

We changed the string, and then we changed the goal. What if we change the board? In the next section, we turn our board into a circle.

\section{Loop-Around Board}

In this section, we change the board, so that the board loops around, but we go back to the original target string SOS. For example, for $n=5$, the board OSEES counts as a win.
This extra twist doesn't change any results that are already wins because this doesn't restrict or change anything about a winning strategy that works, but some draws can turn into wins because it's much easier to create a losing pair with a loop-around board. Thus, we need to check odd boards of sizes 3 and 5 and even boards of sizes less than 16.

Suppose $n=3$, the first player places an S. Then, on the next turn, the first player plays an opposite letter to what the second player played. Thus, the board has two S's and one O. This is a win for the first player. We leave it as an exercise for the reader to check that the boards of sizes 4, 5, 6, and 8 are a draw.

\begin{theorem}
In the SOS game on a circular board of size $n$, the first player wins when $n \geq 7$ and odd. The second player wins when $ \geq 10$ and even. The first player wins for $n=3$. Otherwise, it is a draw.
\end{theorem}

\begin{proof}
The concept of a losing pair holds in this case. And the proof that the losing pair leads to a win depending on parity also works. That means a winning strategy for a standard game of SOS continues working here. We only need to check the small boards that are drawn. We assume that the reader finished the exercise above and checked that boards of sizes below 9. What is left to do is to study even boards of sizes 10, 12, and 14.

Suppose the board is even. If the first player places an S, then the second player can create a losing pair on any board of size at least 4. Thus, the first player places an O. Suppose $n > 8$. Then excluding O, placed by the first player and the two neighboring cells, we have at least 7 cells in a row left. The second player can place S in the middle and will be able to create a losing pair on the next step. Thus, the second player wins. This concludes the proof.
\end{proof}

In this section, we changed the board's shape, but what if we added another dimension to the board?

\section{A 2-by-$n$ Board}

Now we play SOS on a 2-by-$n$ board.

Though this game is 2-dimensional, it is actually simple because there are no possible vertical or diagonal lines that could result in an SOS. Thus, even if the rules might allow vertical or diagonal wins, only horizontal wins can happen here. On a 2-by-$n$ board, the total amount of squares is always even, which means, given our intuition for the one-dimensional game, player 2 might have an advantage. 

\begin{theorem}
The SOS game on a 2-by-$n$ board is a draw for $n < 7$ and a win for the second player otherwise.
\end{theorem}

\begin{proof}
Suppose $n < 7$. For the second player obtaining a draw is very simple as the player can separate the game into 2 separate 1-by-$n$ games, which are both draws as proven in Theorem~\ref{thm:SOS}. This can be done by making moves only in the same row where the first player made the previous move. 

The first player can place an O on the edge of one of the rows and, after that, play only in the row where the second player played his last move. This way, the first player essentially plays two games on two boards, 1-by-$n$ and 1-by-$(n-1)$. Thus, by Theorem~\ref{thm:SOS}, the first player can guarantee a draw too.

Thus, both players can prevent the other player from winning for $n < 7$.

Suppose $n \geq 7$. If a losing pair is created, then the game is won by player 2 on any 2-by-$n$ board as it has an even number of squares, and thus, the first player would be the one to place something inside one of the losing pairs first. This means that all the second player must do to win is to create a losing pair. This can be done by putting an S in the middle of the row the first player didn't go in. Then, on the second turn, the second player is guaranteed to be able to make a losing pair by placing an S on the side of the S placed on the first turn that the first player hasn't touched, thus creating a losing pair. The second player should always have a safe move for the same reason there is a safe move for the second player on an even 1-dimensional board. 

Since the second player is able to make a losing pair for any $n \geq 7$ and always has a safe move, the second player wins on such boards.
\end{proof}

We can go further. On any 2-by-$n$ board, as long as the target string is longer than 1 letter, player 2 will never lose. Indeed, player 2 can implement the following strategy:
\begin{itemize}
\item If player 2 can win, he does so;
\item Otherwise, player 2 can repeat player 1's last move in the other row.
\end{itemize}
This works because if player 1 plays a safe move, player 2 copies the safe move in the other row. If player 1's move is not safe, then player 2 wins.

We covered loop-around boards and 2d boards, but we had a crazy idea: What if the board is changing with each move?

\section{Expandable SOS}

In this game, the grid expands with each turn. It starts as a 1-by-1 board, and the first player places their letter. After the odd turns, one row at the bottom is added. After even turns, one column to the right is added.

In this version, the first player always wins if they play optimally. We leave it as a simple exercise for the reader to show that the second player wins if the first player starts with an S. Now, we show the winning strategy for the first player.

The first player places an O. We divide the rest into two cases.

\textbf{Case 1.} Suppose the second player places an S underneath. Then, the first player places an O to the right of the S. It doesn't matter what the second player does next. On the fifth move, as soon as the third column appears, the first player completes the SOS in the second row. Table~\ref{table:case1} shows the state of the game after the third and fifth moves, respectively. Question marks show where the second player might have placed a letter.

\begin{table}[ht!]
\begin{center}
\begin{tabular}{|c|c|}
\hline
O &  \\ \hline
S & O \\ \hline
\end{tabular}
\quad
\begin{tabular}{|c|c|c|}
\hline
O & ? & \\ \hline
S & O & S \\ \hline
? & ? & \\ \hline
\end{tabular}
\end{center}
\caption{The game's progress in the first case.}
\label{table:case1}
\end{table}

\textbf{Case 2.} Suppose the second player places an O.  Then, the first player places an S in the first row, as seen in Table~\ref{table:case2}, on the left. It doesn't matter what is played on the fourth move, as that move can't be a win. The first player wins if the second player makes an unsafe move on the fifth move. Otherwise, the first player places an O in the first row third column, creating a brink; see Table~\ref{table:case2} middle. But the second player can't finish the SOS as the fourth cell in the first row is unavailable. The second player can't win anywhere, as the previous move was safe, and new cells that opened up don't help. The first player wins on the next move; see Table~\ref{table:case2} right.

\begin{table}[ht!]
\begin{center}
\begin{tabular}{|c|c|}
\hline
O & S\\ \hline
O &  \\ \hline
\end{tabular}
\quad
\begin{tabular}{|c|c|c|}
\hline
O & S & O \\ \hline
O & ? &  \\ \hline
O & ? &  \\ \hline
\end{tabular}
\quad
\begin{tabular}{|c|c|c|c|}
\hline
O & S & O & S \\ \hline
O & ? & ? &  \\ \hline
O & ? & ? &  \\ \hline
S & ? & ? & \\ \hline
\end{tabular}\\
\caption{The game's progress in the second case.}
\label{table:case2}
\end{center}
\end{table}

We started by discussing an example of this game and ended by proving the following theorem.

\begin{theorem}
The expandable SOS game is a win for the first player.
\end{theorem}

However crazy this idea was, we have yet another idea: What if the players were limited in placing the letters in certain squares depending on the previous move?

\section{Restricted-Placement SOS}

The idea of this game comes from the ultimate game of tic-tac-toe. Each move restricts where the next player can move, so we call this game the \textit{restricted-placement SOS}.

We play on a 3-by-3 square grid. The first player can start in any position. The next player has to play in the row corresponding to the previous move's column. The player who finishes SOS in a row, column, or diagonal wins.

This game differs from other games so that a player sometimes is unable to move despite having empty cells on the board. We provide an example in the next paragraph.

Suppose the first player places an S in row 1, column 1. The second player has to move in the first row and places an S in row 1, column 2. The first player needs to go in the second row. So the first player places an S in row 2, column 1. The second player has to move in the first row, where only one cell is left. So the second player places an S in row 1, column 3. Now, the first player must go in the third row and places an S in row 3, column 1. The second player must play a move in the first row, but it is completely filled.

However, with the correct play, the first player always wins.

\begin{theorem}
In the restricted-placement SOS game, the first player wins.
\end{theorem}

\begin{proof}
The first player begins by putting an S in the top left corner. Thus, the second player has to move in the first row. There are four possibilities for the second move.

\textbf{Case 1.} The second player doesn't create a brink in row 1. He either places an S in row 1, column 2, or places an O in row 1, column 3. Then the first player has to play in the second or third row. The first player places a letter in column 1, creating a brink there. Then the second player has to place anything in the last empty cell of the first row. The second player couldn't win, as there was no brink. After that, the first player can finish SOS in the first column. Table~\ref{table:case1rp} shows the game's progress in this case, where X stands for either S or O.

\begin{table}[ht!]
\begin{center}
\begin{tabular}{|p{0.25cm}|p{0.25cm}|p{0.25cm}|}
\hline
   S  & S & \\ \hline
      &   & \\ \hline
      &   & \\ \hline
\end{tabular}
\quad
\begin{tabular}{|p{0.25cm}|p{0.25cm}|p{0.25cm}|}
\hline
   S  & S & \\ \hline
   O  &  & \\ \hline
      &   & \\ \hline
\end{tabular}
\quad
\begin{tabular}{|p{0.25cm}|p{0.25cm}|p{0.25cm}|}
\hline
   S  & S & X\\ \hline
   O  &  & \\ \hline
      &   & \\ \hline
\end{tabular}
\quad
\begin{tabular}{|p{0.25cm}|p{0.25cm}|p{0.25cm}|}
\hline
   S  & S & X\\ \hline
   O   &   & \\ \hline
   S   &   & \\ \hline
\end{tabular} 
\caption{The game's progress in the first case.}
\label{table:case1rp}
\end{center}
\end{table}

\textbf{Case 2.} The second player creates a brink in row 1. He either places an S in row 1, column 3, or places an O in row 1, column 2. Then the first player has to play in the second or third row. The first player places a letter on the main diagonal, creating a brink there. The second player can't win on their move as there is only one letter in the row where the player needs to go. If the second player places a letter in the first column, the first player can finish the brink in the first row. If the second player places a letter in the other column, the first player can finish the diagonal brink. Table~\ref{table:case1rp} shows an example of the game in this case.

\begin{table}[ht!]
\begin{center}
\begin{tabular}{|p{0.25cm}|p{0.25cm}|p{0.25cm}|}
\hline
   S  & O & \\ \hline
      &   & \\ \hline
      &   & \\ \hline
\end{tabular}
\quad
\begin{tabular}{|p{0.25cm}|p{0.25cm}|p{0.25cm}|}
\hline
   S  & O & \\ \hline
      & O & \\ \hline
      &  & \\ \hline
\end{tabular}
\quad
\begin{tabular}{|p{0.25cm}|p{0.25cm}|p{0.25cm}|}
\hline
   S  & O & \\ \hline
      & O & X\\ \hline
      &   & \\ \hline
\end{tabular}
\quad
\begin{tabular}{|p{0.25cm}|p{0.25cm}|p{0.25cm}|}
\hline
   S  & O & \\ \hline
      & O & X\\ \hline
      &   & S\\ \hline
\end{tabular}
\caption{The game's progress in the second case.}
\label{table:case2rp}
\end{center}
\end{table}
\end{proof}

After exploring all these variations on SOS, we can SSS-afely claim that we had SOO much fun with this game! We encourage the reader to SSSS-ee what other fun variations they can uncover and SSOO-lve!

\section{Acknowledgments}

This project was done as part of MIT PRIMES STEP, a program that allows students in grades 6 through 9 to try research in mathematics. Tanya Khovanova is the mentor of this project. We are grateful to PRIMES STEP for this opportunity.

\end{document}